\documentclass{article}
\usepackage{amssymb}
\usepackage{amsthm}
\usepackage{amsmath}
\usepackage{graphicx}

\newcommand{\R}{\mathbb R}

\newtheorem{lemma}{Lemma}
\newtheorem{example}{Example}

\newenvironment{potwr}[1]{{\noindent\it \underline{Proof of Theorem \ref{#1}}}:}{\qed}

\newtheorem{definition}[lemma]{Definition}
\renewcommand{\Im}{\operatorname{Im}}
\newcommand{\dist}{\operatorname{dist}}

\begin{document}
\title{A Morse deformation lemma at infinity}
\author{J. Haddad}
\maketitle
\abstract{We prove a generalized version of the classic deformation lemma from Morse Theory that considers functions going to $-\infty$ at a compact set, and allowing the lower value of the deformation to be $-\infty$. The result is valid for a class of functions satisfying a suitable growth condition.}

\section{Introduction and main result}
Given a smooth manifold $M$ and a smooth function $f:M \to \R$, Morse Theory provides a relation between the topology of the sets $f^{-1}((-\infty, a])$ and the critical points of $f$, this is, points $x \in M$ such that $d_x f = 0$.
We recall that the singular values of $f$ are the images by $f$ of its critical points and the regular values are those that are not singular.

Morse Theory was developed around 1926 by Marston Morse and has had a great number of generalizations and applications in varied contexts, including for example Morse Theory for functions defined on Hilbert spaces \cite{P}, Morse Homology \cite{B}, Discrete Morse Theory \cite{F}, etc.
In almost all of its ``flavors'', classic Morse Theory relies on a few number of fundamental lemmas, one of wich is the following:

\begin{lemma}
\emph{(Morse Deformation Lemma)}

\label{deformation_lemma_0}
Let $M$ be a compact smooth manifold and $g: M \to \R$ a $C^2$ function such that all values in a closed interval $[a,b]$ are regular.
Then $g^{-1}((-\infty, a] )$ is a strong deformation retract of $g^{-1}((-\infty,b])$.
\end{lemma}

The compacity assumption on $M$ can be replaced by the compacity of the set $g^{-1}((-\infty,b])$.
Sourprisingly, it is not found anywhere in the literature any generalization of Lemma \ref{deformation_lemma_0} allowing the lower value $a$ to be $-\infty$.
More precisely, let $f: M \setminus \Sigma \to \R$ be a $C^2$ function defined outside a compact set $\Sigma$, for which $f(x_n) \to -\infty$ whenever $(x_n)$ is a sequence in $M \setminus \Sigma$ converging to $z \in \Sigma$, and such that all values in $(-\infty, b]$ are regular.
We ask if there can be established a topological equivalence between $\Sigma$ and a set of the form $f^{-1}((-\infty,b]) \cup \Sigma$.
This problem is easy to reduce to a generalization of Lemma \ref{deformation_lemma_0} where the value $a$ is allowed to be singular. This reduction is done by composing $g$ with a function $\Phi:\R \to \R$ satisfying $\lim_{x \to -\infty} \Phi(x) = a$.

Some approaches in this directions have been made.
A degenerate version of Lemma \ref{deformation_lemma_0} is proved in \cite{KCC} where $a$ (and not $b$) is allowed to be a singular value, under the requirement that the level set $f^{-1}(\{a\})$ is totally disconnected.
The proof can be traced back to \cite{COR} where a Morse Theory for continuous functions is established. In such a theory, the critical points must be isolated.
Morse-Bott Theory deals with functions having manifolds consisting entirely of critical points.
These so called ``critical manifolds'' must be non-degenerate in the sense that at each point of the critical manifold, the Hessian restricted to the normal tangent space is non-degenerate.
A critical point is regarded as a $0$-dimentional critical manifold, generalizing the concept of non-degenerate critical point in a very elegant and natural way.

In contrast, our deformation lemma does not require the $\Sigma$ ``level set'' to be even a manifold.
The non-degeneracy is replaced by a condition in the growth of the function when it tends to $-\infty$.
While this condition restricts the class of functions we may consider, our deformation lemma is still more general than the mentioned ones, for our purpose.

With this in mind, we give the following definition.
\begin{definition}
\label{def_fast_decreasing}
Let $M$ be a smooth riemannian manifold and $g: M \setminus \Sigma \to \R$ be a $C^1$ function defined outside a compact set $\Sigma$.
We say that $g$ is a ``fast decreasing'' function at $\Sigma$ under level $b \in \R$ if
\[\lim_{x \to z} g(x) = -\infty\]
for every $z \in \Sigma$, and 
\[g(x) \leq b \quad \Rightarrow \quad \|\nabla g(x)\| > 1/ \phi(g(x))\]
where $\phi:(-\infty, b] \to \R$ is a strictly positive continuous function with bounded primitive.

\end{definition}
Observe that all values in the interval $(-\infty, b]$ of a fast decreasing function under level $b \in \R$, are regular.

Our main result is the following:
\begin{lemma}
\label{deformation_lemma_1}
Let $M$ be a smooth riemannian manifold, $\Sigma \subset M$ a compact set and $g: M \setminus \Sigma \to \R$ a $C^2$ fast decreasing function at $\Sigma$ under level $b \in \R$.
Assume also that $g^{-1}((-\infty,b]) \cup \Sigma$ is compact.
Then $\Sigma$ is a strong deformation retract of $g^{-1}((-\infty,b]) \cup \Sigma$.
\end{lemma}

Before giving the proof we review some examples:
\begin{example}
\normalfont
Let $M = \R^2$ and 
\[f(x,y) = - \left( \left(x^2+y^2\right)^3-4 x^2 y^2 \right)^{-2}\]
The singular set for $f$ pictured in Figure \ref{quadrifolium} is the four-leaved clover (see \cite{G} pages 92, 93).
We verify the fast decreasing condition as follows:

Let $p(x,y) = \left(x^2+y^2\right)^3-4 x^2 y^2$. We compute for $(x,y)$ in a small neighbourhood of $0$,
\begin{align*}
p(x,y)^2 		&\leq C \|(x,y)\|^{12} + 16 x^4 y^4\\
\|\nabla p(x,y)\|^2	&= 36 (x^2+y^2)^5 + x^2 y^2 (x^2+y^2)(64-192(x^2+y^2)) \\
			&\geq C \|(x,y)\|^{10} + C x^2 y^2 \|(x,y)\|^2\\
\|\nabla p(x,y)\|^{2(1+\frac 15)} 	&\geq C \|(x,y)\|^{10+10\frac 15} + C x^{2+2\frac 15} y^{2+2\frac 15} \|(x,y)\|^{2+2\frac 15}\\
					&\geq C \|(x,y)\|^{12} + C x^{3+3\frac 15} y^{3+3\frac 15}
\end{align*}
Where $C>0$. Then we have
\begin{align}
\label{fp_ineq}
\|\nabla p(x,y)\|^{1+\frac 15} 	&\geq C |p(x,y)|.
\end{align}

Since no point in $\Sigma \setminus \{0\}$ is singular, there exists $K$ a compact neighbourhood of $\Sigma$ such that inequality \eqref{fp_ineq} is satisfied for some $C>0$ and all $(x,y) \in K \setminus \Sigma$.
Finally, inequality \eqref{fp_ineq} implies
\[\|\nabla f(x,y)\| \geq \frac 1 {|f(x,y)|^{-\frac {13}{12}}}\]
which is the fast decreasing condition, valid in a neighbourhood of $\Sigma$, this is, for $f(x,y)$ sufficiently negative.
\begin{figure}[h]
\begin{center}
\label{quadrifolium}
\includegraphics[scale=0.5]{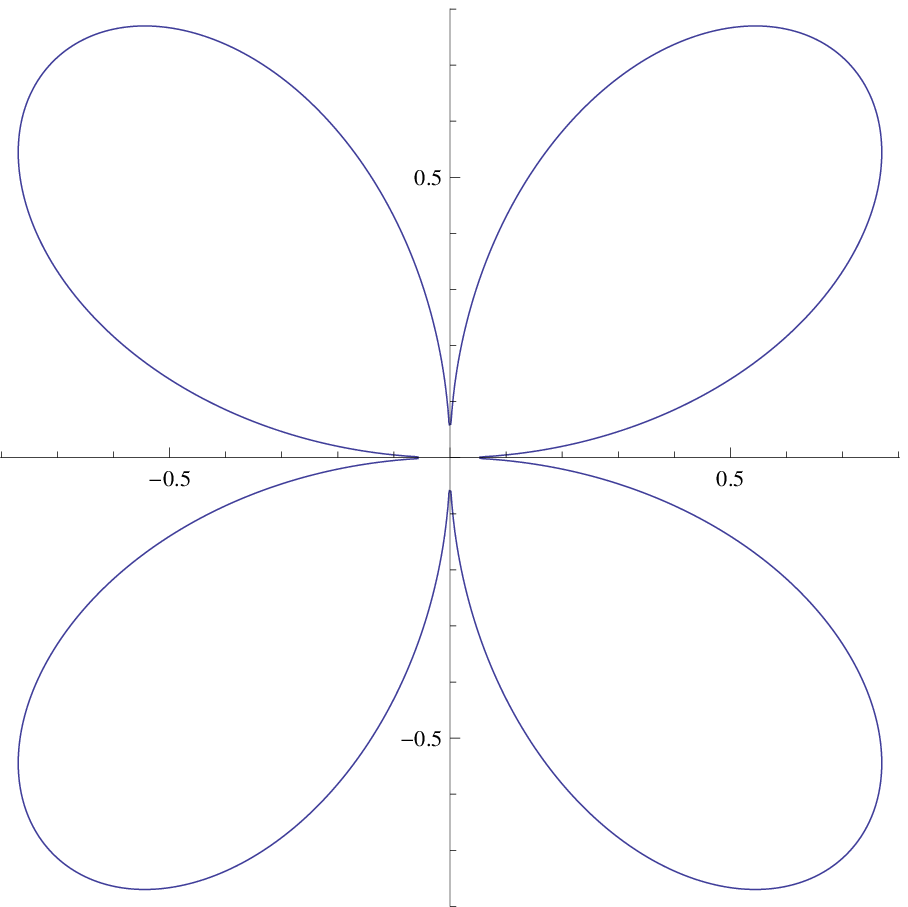}
\caption{The singular set of $f$.}
\end{center}
\end{figure}
Lemma \ref{deformation_lemma_1} shows that the sets $f^{-1}((-\infty, b]) \cup \Sigma$ for $b\ll 0$ have the homotopy type of $S^1 \wedge S^1\wedge S^1 \wedge S^1$.
Although the example is artifitial, this fact cannot be proved directly using the classic version of the Deformation Lemma.
\end{example}

\begin{example}
\normalfont
In \cite{AH1} and \cite{AH2} the authors need to study the level sets of the function $G:\R^n \to \R$ given by
\[G(x) = - \int_{S^1} \frac {dt} {\|x - \Gamma(t) \|}\]
where $\Gamma:S^1 \to \R^n$ is a $C^1$ embedding.
Specifically, it is shown that $G^{-1}((-\infty, b] ) \cup \Sigma$ is a tubular neighbourhood of $\Sigma = \Im(\Gamma)$ for $b \ll 0$.
It is not hard to see that $G$ is a fast decreasing function, thus Lemma \ref{deformation_lemma_1} shows that $G^{-1}((-\infty, b] ) \cup \Sigma$ has the homotopy type of $S^1$.
\end{example}

\section{Proof of the lemma}
Lemma \ref{deformation_lemma_1} is a consequence of the following
\begin{lemma}
\label{deformation_lemma_2}
Let $M$ be a smooth riemannian manifold and  $f:M \to [0, \infty)$ be a continuous function which is $C^2$ in $M \setminus \Sigma$ where $\Sigma = f^{-1}(0)$, and such that all values in an interval $(0,a]$ are regular. 

Assume $\|\nabla f(x)\| > \alpha $ for some constant $\alpha > 0$ for all $x \in M \setminus \Sigma$, and assume that $f^{-1}([0,a])$ is compact.
Then $\Sigma$ is a strong deformation retract of $f^{-1}([0,a])$ and $f^{-1}([a, \infty))$ is a strong deformation retract of $M \setminus \Sigma$.
\end{lemma}
\begin{proof}
As in the classic proof of the Morse deformation lemma take the vector field 
\[X_p = - \eta(p) \frac {\nabla f(p)}{\|\nabla f(p)\|^2}\]
where $\eta:M \to \R$ is a cutoff function with compact support and such that $\eta(x) = 1$ for every $x \in f^{-1}([0,a])$.
The vector field is defined and $C^1$ at $M \setminus \Sigma$.
The corresponding flow $\varphi(x,t)$ is defined and $C^1$ in an open subset of $M \times \R$.
Also, since $X$ has compact support and since $f(\varphi(x,t)) = f(x) - t$, we have that $\varphi$ is well defined in the open set
\[\left\{(x,t) \in M \times \R /\; x \in M \setminus \Sigma, \hbox{ and } t \in [0,f(x))\right\}.\]
Consider $\psi(x,t) = \varphi(x, t.f(x))$ which is continuous at $(M \setminus \Sigma) \times [0,1)$.

For $x \in M \setminus \Sigma$ and $t_1, t_2 \in [0,1)$ the points $\varphi(x,t_2), \varphi(x,t_1)$ are joined by the integral curve $\gamma(t) = \varphi(x,t)$ so
\[\dist(\varphi(x,t_2), \varphi(x, t_1)) \leq \int_{t_1}^{t_2} \|\gamma'(t)\| dt.\]
where $\dist$ is the geodesic distance on $M$ and $\gamma'(t) = X_{\gamma(t)}$. Then by the hypothesis on the gradient of $f$ we have 
\begin{align}
\nonumber
\dist(\varphi(x,t_2), \varphi(x, t_1)) &\leq \alpha^{-1} |t_2 - t_1|\\
\label{psi_lip}
\dist(\psi(x,t_2), \psi(x, t_1)) &\leq \alpha^{-1} f(x) |t_2 - t_1|.
\end{align}

We shall prove first that $\psi$ extends continuously to $M \times [0,1)$:
Let $(x_n,t_n)$ be a sequence in $(M \setminus \Sigma) \times [0,1)$ converging to $(z, t) \in \Sigma \times [0,1)$.
We have $f(x_n) \to 0$ and 
\begin{align*}
\dist(\psi(x_n,t_n), x_n) &= \dist(\psi(x_n,t_n), \psi(x_n, 0))\\
&= \dist(\varphi(x_n,t_n f(x_n)), \varphi(x_n, 0))\\
&\leq \alpha^{-1} |t_n f(x_n)|\\
&\leq \alpha^{-1} f(x_n) \to 0
\end{align*}
 so $\psi(x_n,t_n) \to z$.
Thus we see that $\psi$ can be extended continuously if we set $\psi(z,t) = z$ for $z \in \Sigma$ and $t \in [0,1)$.

Now we must see that $\psi$ extends continuously to $M \times [0,1]$.
Take $x_0 \in M$ and consider the function $\psi_{x_0}$ defined by $ \psi_{x_0}(t) = \psi(x_0,t)$.
From \eqref{psi_lip} we see that $\psi_{x_0}$ can be extended continuously to $t \in [0,1]$ and this allow us to extend $\psi$ to $M \times [0,1]$ as
\[\psi(x_0, 1) = \lim_{t \to 1^-} \psi(x_0,t). \]
Lets see that this extension is continuous:

Clearly, with the new definition of $\psi$, inequality \eqref{psi_lip} remains true for any $x \in M, t_1, t_2 \in [0,1]$.
Assume by contradiction that there is a sequence $(x_n, t_n)$ in $M \times [0,1]$ converging to $(x_0, 1)$, such that 
\[\dist(\psi(x_n, t_n), \psi(x_0, 1)) \geq \epsilon > 0.\]

Take $s = 1- \frac \varepsilon {3 \alpha^{-1} (f(x_0)+1)}$. We have
\[\dist( \psi(x_0, 1), \psi(x_0,s))\leq f(x_0). \alpha^{-1}.(1-s) < \varepsilon /3\]
and for $n$ large, $t_n > s$ and
\begin{align*}
\dist(\psi(x_n, t_n), \psi(x_n,s)) &\leq (f(x_0) + 1). \alpha^{-1}.(t_n - s)\\
&\leq (f(x_0) + 1). \alpha^{-1}.(1 - s) < \varepsilon /3\\
\dist(\psi(x_0, s), \psi(x_n,s)) & < \varepsilon/3
\end{align*}
then we have a contradiction.
\end{proof}
\begin{potwr}{deformation_lemma_1}
Let $\Phi:(-\infty, b] \to \R$ be a bounded primitive of $\phi$. Since $\Phi$ is monotone increasing and bounded, we may assume that 
\[\lim_{x\to -\infty} \Phi(x) = 0.\]
Consider the function $f(x) = \Phi(g(x))$. 
It is positive and can be extended to $\Sigma$ as a continuous (not necessarily smooth) function.
The chain rule gives
\[\nabla f(x) = \phi(g(x)) . \nabla g(x)\]
\[\|\nabla f(x)\| \geq 1\]
for all $x \in M \setminus \Sigma$ such that $g(x) \in (-\infty, b]$.
This proves that all values in the interval $(0, \Phi(b)]$ are regular and that the function $f$ is in the conditions of Lemma \ref{deformation_lemma_2}.
\end{potwr}


\begin{thebibliography}{4}

\bibitem {KCC}
K.C. Chang - Methods in Nonlinear Analysis, Springer Monographs in Mathematics (2005)
\bibitem {COR}
J.N. Corvellec, Morse Theory for Continuous Functionals, Journal of Mathematical Analysis and Applications, Volume 196, Issue 3, 1995, Pages 1050-1072, ISSN 0022-247X
\bibitem {AH1}
J. Haddad, P. Amster, Critical point theory in knot complements, Differential Geometry and its Applications, Volume 36, October 2014, Pages 56-65, ISSN 0926-2245, http://dx.doi.org/10.1016/j.difgeo.2014.07.005.
\bibitem {AH2}
J. Haddad, P. Amster, On existence of periodic solutions for Kepler type problems, Topol. Methods Nonlinear Anal. (2016). DOI: 10.12775/TMNA.2016.053
\bibitem {G}
C. G. Gibson, Elementary Geometry of Algebraic Curves, An Undergraduate Introduction, Cambridge University Press, Cambridge, 2001, ISBN 978-0-521-64641-3.
\bibitem{F}
R. Forman, Morse Theory for Cell Complexes, Advances in Mathematics 134, 90 145 (1998).
\bibitem{P}
R.S. Palais, Morse Theory on Hilbert Manifolds, Topology Vol. 2, pp. 299-340, ISSN 0040-9383, (1963).
\bibitem{B}
A. Banyaga, D. Hurtubise, Lectures on Morse Homology, Kluwer Texts in the Mathematical Sciences, Vol. 29 (2004), ISBN 978-1-4020-2696-6,

\end{thebibliography}
\end{document}